\documentclass{amsart}
\usepackage[T1]{fontenc}

\usepackage{amssymb}
\usepackage{mathtools}
\usepackage{enumitem}

\usepackage{tikz}
\usepackage{tkz-euclide}
\usetikzlibrary{math,calc,decorations.markings} 
\usepackage[top=1in,bottom=1in,right=1in,left=1in]{geometry}
\usepackage{pgfplots}
\pgfplotsset{compat=1.16}

\newcommand{\rr}{\mathbb{R}}
\newcommand{\nn}{\mathbb{N}}
\newcommand{\qq}{\mathbb{Q}}
\newcommand{\zz}{\mathbb{Z}}
\newcommand{\cc}{\mathbb{C}}
\newcommand{\dd}{\mathbb{D}}

\newcommand{\eps}{\varepsilon}

\newcommand{\fff}{{\mathcal F}}

\newcommand{\ppp}{{\mathcal P}}

\newcommand{\vvv}{{\mathcal V}}

\newcommand{\dc}{\coloneqq}
\newcommand{\isp}[1]{\quad\text{#1}\quad}

\newcommand{\del}{\partial}
\DeclareMathOperator{\Arg}{Arg}
\DeclareMathOperator{\res}{Res}

\newtheorem{theorem}{Theorem}
\newtheorem*{theorem*}{Theorem}
\newtheorem{corollary}{Corollary}
\newtheorem{lemma}{Lemma}
\newtheorem{prop}{Proposition}
\theoremstyle{remark}\newtheorem{remark}{Remark}

\title{On the Cauchy transform of complex powers of the identity function}
\author{Benjamin Faktor, Michael Kuhn, Gahl Shemy}

\address{University of California Santa Barbara, Santa Barbara CA 93106}
\email{benjaminfaktor@ucsb.edu}

\address{University of California Santa Barbara, Santa Barbara CA 93106}
\email{mkuhn@ucsb.edu}

\address{University of California Santa Barbara, Santa Barbara CA 93106}
\email{gahlshemy@ucsb.edu}


\begin{document}
    \maketitle
    \vspace{-0.8cm}
    \begin{abstract}
        The integral $\int_{|z|=1} \frac{z^\beta}{z-\alpha} dz$ for $\beta=\frac{1}{2}$ has been comprehensively studied by Mortini and Rupp for pedagogical purposes. We write for a similar purpose, elaborating on their work with the more general consideration $\beta \in \mathbb{C}$. This culminates in an explicit solution in terms of the hypergeometric function for $|\alpha| \neq 1$ and any $\beta \in \cc$. For rational $\beta$, the integral is reduced to a finite sum. A differential equation in $\alpha$ is derived for this integral, which we show has similar properties to the hypergeometric equation.
    \end{abstract}

    %\newpage
    \section{Introduction}
        The purpose of this paper is to investigate integrals of the form
        \begin{equation}\label{main-integral}
            \int_{|z|=1} \frac{z^\beta}{z-\alpha}\;dz.
        \end{equation}
        Our personal interest in this type of integral stems from a recent paper due to Mortini and Rupp \cite{mortini-rupp}, in which the authors evaluate \eqref{main-integral} for $\beta=\tfrac{1}{2}$ using various methods.
        
        Initially we note that the function $z^\beta$ must be defined, for general $\beta\in\cc$, in terms of some branch of the complex logarithm. In our notation, for $0<\theta<2\pi$, $\log_\theta(z)$ will represent the branch of the complex logarithm with branch cut $\{re^{i\theta} :  r \ge 0\}$; it is defined on the simply connected domain $\cc\setminus\{re^{i\theta}: r \ge 0\}$, and we fix $\log_\theta(1) = 0$.  Under these conditions our branch is
        $$
            \log_\theta(z) = \ln|z| + i\arg_\theta(z)
        $$
        where $\arg_\theta$ is the argument function with values in $(\theta-2\pi,\theta)$.\\
        This branch can be related to the branch of the square root discussed in \cite{mortini-rupp} by taking $t_0=\theta-2\pi$.\\
        We denote by $\Arg(z)$ the argument of $z$ falling in the range $[0,2\pi)$, and by $\arg(z)$ the equivalence class (modulo $2\pi$) of all possible values for the argument of $z$. Any condition with $\arg(z)$ is considered satisfied if one representative satisfies the condition.
        
        The implications of using a branch of the complex logarithm to define the complex power are that even when we choose $|\alpha|\neq 1$, the meromorphic function
        \begin{equation}\label{m-def}
            m_{\alpha,\beta,\theta}(z) := \frac{z^\beta}{z-\alpha} = \frac{e^{\beta\log_\theta(z)}}{z-\alpha}
        \end{equation}
        will not be analytic, or even \textit{continuous}, on the boundary of the unit disk. This is due to the branch cut necessary for the $\log_\theta$ function used in \eqref{m-def}. The discontinuity at the branch cut, although merely a jump, prevents a simple evaluation with direct application of Cauchy's Residue Theorem. Rather, one must proceed using different methods.\\
        
        The main results of the paper are explicit expressions of  \eqref{main-integral} in the two cases of $|\alpha| > 1$ and $|\alpha| < 1$.  Specifically, we prove: \begin{theorem}\label{thm}
            When $|\alpha|>1$,
            \begin{equation}
                \int_{\del\dd} m_{\alpha,\beta,\theta} = 
                \begin{cases}
                   -2\pi i\alpha^{\beta}\qquad\qquad\qquad\qquad\qquad\qquad\qquad\qquad\qquad\qquad\beta\in\zz_{<0},\\
                   0\qquad\qquad\qquad\qquad\qquad\qquad\qquad\qquad\qquad\qquad\qquad\;\;\;\beta=0,\\
                   e^{i\beta\theta}\left(1-e^{-2\pi i\beta}\right)\frac{1}{\beta}\left[1-\,_2F_1(1,\beta;1+\beta;\alpha^{-1}e^{i\theta})\right]\qquad\;\,\beta\in\cc\setminus\zz_{\leq0}.\notag
                \end{cases}
            \end{equation}
            
            When $|\alpha|<1$,
            \begin{equation}
                \int_{\del\dd} m_{\alpha,\beta,\theta} = 
                \begin{cases}
                    2\pi i\alpha^\beta \qquad\qquad\qquad\qquad\qquad\qquad\qquad\qquad\qquad\beta\in\zz_{\geq0},\\
                    e^{i\beta\theta}\left(1-e^{-2\pi i\beta}\right)\frac{1}{\beta}\,_2F_1(1,-\beta;1-\beta;\alpha e^{-i\theta})\qquad\;\beta\in\cc\setminus\zz_{\geq0}.\notag
                \end{cases}
            \end{equation}
        \end{theorem}
        
        In \S\ref{contour-1}, the unit circle is approximated with a contour of integration which avoids the branch cut in order to derive an equation involving \eqref{main-integral}. The connection between \eqref{main-integral} and the hypergeometric function is made in \S\ref{hypergeo-section} through the identification of a core integral in \S\ref{core}. In \S\ref{series-method-section}, series manipulation leads to the proof of Theorem \ref{thm}. The particular case when $\beta \in \qq\setminus\zz$ is further simplified in \S\ref{m/n-sec}, and in \S\ref{diffeq} we include a derivation of a differential equation for which \eqref{main-integral} is a solution.

    \section{Contour Method} \label{contour-1}
        We first extend \S1 in \cite{mortini-rupp}, evaluating \eqref{main-integral} via contour integration.  For this section alone (\S\ref{contour-1}) it is additionally assumed that $\Arg(\alpha)\neq\theta$ and $\alpha\neq0$, so that $\alpha$ does not lie on the branch cut. Furthermore, we assume that $\Re(\beta)>0$, as this condition will be necessary for certain bounds.  The purpose of this section is to prove the following lemma:
        \begin{lemma}\label{contour-lemma}
            If $\Arg(\alpha)\neq\theta$, and $\Re(\beta)>0$, then for $0<|\alpha|<1$,
            \begin{equation}
                \int_{\del\dd}m_{\alpha,\beta,\theta} = 2\pi i \alpha^\beta + e^{i\beta\theta}(1-e^{-2\pi i\beta})\int_0^1\frac{e^{\beta\ln{t}}}{t-\alpha e^{-i\theta}}\;dt,\notag
            \end{equation}
            and for $|\alpha|>1$,
            \begin{equation}
                \int_{\del\dd}m_{\alpha,\beta,\theta} = e^{i\beta\theta}(1-e^{-2\pi i\beta})\int_0^1\frac{e^{\beta\ln{t}}}{t-\alpha e^{-i\theta}}\;dt.\notag
            \end{equation}
            \begin{proof}
                There are 3 main steps:
                \begin{enumerate}
                    \item[\S\ref{contour-construction})] constructing a proper contour;
                    \item[\S\ref{res-app})] finding singularities and computing their residues;
                    \item[\S\ref{res-app})] using limits to derive a useful equation.
                \end{enumerate}
                The lemma follows from plugging \eqref{lim-lhs-contour-integral}, \eqref{lim-L-contour-integral}, \eqref{lim-C-contour-integral}, \eqref{lim-M-contour-integral}, and \eqref{lim-D-contour-integral} all back into \eqref{lim-contour-integral}.
            \end{proof}
        \end{lemma}

        \subsection{Constructing the Contour} \label{contour-construction}
           Take a branch of the complex logarithm $\log_\theta$ in the definition of $z^\beta$, and let the contour of integration $\Gamma_{\eps,\theta,\rho}$ consist of:
           \begin{enumerate}[label = \alph*)]
                \item the line segment $L_{\eps,\theta,\rho} \dc \{z \in \cc: \rho \le |z| \le 1, \arg z = \theta + \eps\}$,
                \item the arc $C_{\eps,\theta} \dc \{z \in \cc: |z| = 1, \theta + \eps \le \arg z \le \theta + 2\pi - \eps\}$,
                \item the line segment $M_{\eps,\theta,\rho} \dc \{z \in \cc: \rho \le |z| \le 1, \arg z = \theta + 2\pi - \eps\}$,
                \item the arc $D_{\eps,\theta,\rho} \dc \{z \in \cc : |z| = \rho, \theta + \eps \le \arg z \le \theta + 2\pi - \eps\}$,
            \end{enumerate}
            oriented as usual, with the bounded region enclosed on the left as we trace the contour. For example, for the principal branch of $\log$ ($\log_\pi$ in our notation), the contour is as in Figure \ref{contour-diagram}.
            \begin{figure}[h!]
                \begin{center}
                    %\begin{tikzpicture}
                        %\node at (-2,2){$\Gamma_{\eps,\pi,\rho}$};
                        %\draw[<->, gray] (-2,0)--(2,0) node[right]{$\Re$};
                        %\draw[<->, gray] (0,-2)--(0,2) node[above]{$\Im$};
                        %\draw[decoration={markings,
                        %mark=at position 0.05 with {\arrow{latex}},
                        %mark=at position 0.2 with {\arrow{latex}},
                        %mark=at position 0.39 with {\arrow{latex}},
                        %mark=at position 0.57 with {\arrow{latex}},
                        %mark=at position 0.78 with {\arrow{latex}},
                        %mark=at position 0.92 with {\arrow{latex}}},
                       % postaction={decorate}] 
                        %(-160:0.5) -- node[below,xshift=0.1cm,yshift=-0.1cm]{$L_{\eps,\pi,\rho}$} (-160:1.75) arc (-160:160:1.75) node[midway,left,yshift=-0.5cm,xshift=-0.1cm]{$C_{\eps,\pi}$} -- node[above,xshift=0.1cm,yshift=0.05cm]{$M_{\eps,\pi,\rho}$} (160:0.5) arc (160:-160:0.5) node[midway,right,yshift=0.5cm,xshift=-0.1cm]{$D_{\eps,\pi,\rho}$};
                    %\end{tikzpicture}
                    \includegraphics{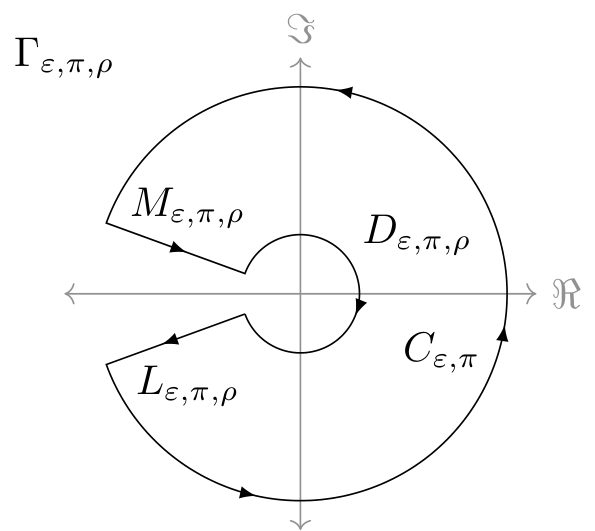}
                    \caption{Contour for $\theta = \pi$.}
                    \label{contour-diagram}
                \end{center}
            \end{figure}
            Under this definition, we have            \begin{equation}\label{contour-integral}
                \int_{\Gamma_{\eps,\theta,\rho}} m_{\alpha,\beta,\theta} = \left[\int_{C_{\eps,\theta}}+\int_{D_{\eps,\theta,\rho}}+\int_{L_{\eps,\theta,\rho}}+\int_{M_{\eps,\theta,\rho}}\right]m_{\alpha,\beta,\theta}.
            \end{equation}
            One can choose any parameterization of the four curves, noting that smooth equivalence of parameterizations will guarantee generality. In particular, we choose
            \begin{enumerate}[label = \alph*)]
                \item $L_{\eps,\theta,\rho}$: $z(t) = te^{i(\theta+\eps)}$ for $\rho\leq t\leq1$,
                    \begin{equation}\label{L-param}
                        \int_{L_{\eps,\theta,\rho}} m_{\alpha,\beta,\theta}(z) \; dz = \int_\rho^1 m_{\alpha,\beta,\theta}(te^{i(\theta + \varepsilon)})e^{i(\theta + \varepsilon)}\;dt;
                    \end{equation}
                \item $C_{\eps,\theta}$: $z(t) = e^{it}$ for $\theta+\eps \leq t \leq \theta+2\pi-\eps$,
                    \begin{equation}\label{C-param}
                        \int_{C_{\eps,\theta}} m_{\alpha,\beta,\theta}(z) \; dz = \int_{\theta+\eps}^{\theta+2\pi-\eps} m_{\alpha,\beta,\theta}(e^{it})\;ie^{it}dt;
                    \end{equation}
                \item $M_{\eps,\theta,\rho}$: $z(t) = te^{i(\theta + 2\pi - \eps)}$ for $1\geq t\geq\rho$,
                    \begin{equation}\label{M-param}
                        \int_{M_{\eps,\theta,\rho}} m_{\alpha,\beta,\theta}(z) \; dz = \int_1^\rho m_{\alpha,\beta,\theta}(te^{i(\theta + 2\pi - \varepsilon)})e^{i(\theta + 2\pi - \varepsilon)}\;dt;
                    \end{equation}
                \item $D_{\eps,\theta,\rho}$: $z(t) = \rho e^{it}$ for $\theta +2\pi - \eps\geq t\geq \theta+\eps$,
                    \begin{equation}\label{D-param}
                        \int_{D_{\eps,\theta,\rho}} m_{\alpha,\beta,\theta}(z) \; dz = \int_{\theta +2\pi - \eps}^{\theta+\eps} m_{\alpha,\beta,\theta}(\rho e^{it})\;i\rho e^{it}dt.
                    \end{equation}
            \end{enumerate}

        \subsection{Applying the Residue Theorem} \label{res-app}
            Applying Cauchy's Residue Theorem requires computing residues for singularities contained within the contour. To compute the residues of the meromorphic function $m_{\alpha,\beta,\theta}(z)$ defined in \eqref{m-def}, first note that $e^{\beta\log_\theta(z)}$ is analytic in $\cc\setminus\{re^{i\theta}\in\cc : r \ge 0\}$, so the only singularity of $m_{\alpha,\beta,\theta}$ is at $\alpha$, and this singularity only becomes relevant when $|\alpha|\leq 1$. This singularity is a simple pole, since
            \begin{equation}\label{residue}
                \lim_{z\to \alpha}(z-\alpha)m_{\alpha,\beta,\theta}(z) = \lim_{z\to \alpha}e^{\beta\log_\theta(z)} = \alpha^\beta \neq 0
            \end{equation}
            but 
            \begin{equation}\label{m-simple}
                \lim_{z\to \alpha}(z-\alpha)^2m_{\alpha,\beta,\theta}(z) = \lim_{z\to \alpha}(z-\alpha)e^{\beta\log_\theta(z)} = 0.\notag
            \end{equation}
            Evaluating as in \eqref{residue}, the residue at $\alpha$ is found to be $\alpha^\beta$.
        
            In order to derive an equation involving \eqref{main-integral}, one might consider first taking the limit $\eps\to0^+$ and then $\rho\to0^+$ in \eqref{contour-integral}:
            \begin{equation}\label{lim-contour-integral}
                \lim_{\rho\to0^+}\lim_{\eps\to0^+}\int_{\Gamma_{\eps,\theta,\rho}} m_{\alpha,\beta,\theta} = \lim_{\rho\to0^+}\lim_{\eps\to0^+}\left[\int_{C_{\eps,\theta}}+\int_{D_{\eps,\theta,\rho}}+\int_{L_{\eps,\theta,\rho}}+\int_{M_{\eps,\theta,\rho}}\right]m_{\alpha,\beta,\theta}.
            \end{equation}

            %\subsubsection{$\Gamma_{\eps,\theta,\rho}$}
                Since the contour $\Gamma_{\eps,\theta,\rho}$ in \eqref{lim-contour-integral} lies in the interior of the simply connected domain of $\log_\theta$ whenever $\eps, \rho > 0$, $m_{\alpha,\beta,\theta}$ is analytic on the path of integration so long as $\alpha$ does not lie on $\Gamma_{\eps,\theta,\rho}$. In this case, Cauchy's Residue Theorem applies and so 
                \begin{equation}
                    \int_{\Gamma_{\eps,\theta,\rho}} m_{\alpha,\beta,\theta} = 2\pi i\;n(\Gamma_{\eps,\theta,\rho},\alpha)\res(m_{\alpha,\beta,\theta},\alpha) = 2\pi i\alpha^\beta\;n(\Gamma_{\eps,\theta,\rho},\alpha) \notag
                \end{equation}
                where $n(\Gamma_{\eps,\theta,\rho},\alpha)$ is the winding number of $\Gamma_{\eps,\theta,\rho}$ around $\alpha$. Note that by definition of the contour, and because $\Arg(\alpha)\neq\theta$ by assumption, we have 
                \begin{equation}\label{winding-number}
                    n(\Gamma_{\eps,\theta,\rho},\alpha) = 
                    \begin{cases}
                        1\isp{if $0<\eps<\min_{\arg(\alpha)}\{|\arg(\alpha)-\theta|\}$ and $0<\rho<|\alpha|<1$,}\\
                        0\isp{otherwise,}
                    \end{cases}
                \end{equation}
                where the notation $\min_{\arg(\alpha)}$ in \eqref{winding-number} denotes that the minimum is taken over all possible representatives of $\arg(\alpha)$.
                It follows that
                \begin{equation}\label{eps0-winding-number}
                    \lim_{\eps\to0^+}n(\Gamma_{\eps,\theta,\rho},\alpha) = 
                    \begin{cases}
                        1\isp{if $0<\rho<|\alpha|<1$,}\\
                        0\isp{otherwise,}
                    \end{cases}\notag
                \end{equation}
                since $\eps$ can certainly be made smaller than $|\arg(\alpha)-\theta|>0$, and that
                \begin{equation}\label{rho0-eps0-winding-number}
                    \lim_{\rho\to0^+}\lim_{\eps\to0^+}n(\Gamma_{\eps,\theta,\rho},\alpha) = 
                    \begin{cases}
                        1\isp{if $0<|\alpha|<1$,}\\
                        0\isp{otherwise}
                    \end{cases}\notag
                \end{equation}
                since $\rho$ can certainly be made smaller than $|\alpha|>0$.
                Therefore               \begin{equation}\label{lim-lhs-contour-integral}
                    \lim_{\rho\to0^+}\lim_{\eps\to0^+}\int_{\Gamma_{\eps,\theta,\rho}} m_{\alpha,\beta,\theta} = 
                    \begin{cases}
                        2\pi i\alpha^\beta\,\isp{if $0<|\alpha|<1$,}\\
                        0\qquad\isp{otherwise.}
                    \end{cases}
                \end{equation}

            %\subsubsection{$C_{\eps,\theta}$}
                In evaluating $\lim_{\rho\to 0^{+}}\lim_{\eps\to 0^{+}}\int_{C_{\eps,\theta}}m_{\alpha,\beta,\theta}$, we use \eqref{C-param} to express \begin{align}
                    \lim_{\rho\to0^+}\lim_{\eps\to0^+}\int_{C_{\eps,\theta}}m_{\alpha,\beta,\theta} &= \lim_{\rho\to0^+}\lim_{\eps\to0^+}\int_{\theta+\eps}^{\theta+2\pi-\eps} m_{\alpha,\beta,\theta}(e^{it})\;ie^{it}dt,\notag\\
                    &= \lim_{\eps\to0^+}\int_{\theta+\eps}^{\theta+2\pi-\eps} m_{\alpha,\beta,\theta}(e^{it})\;ie^{it}dt,\notag\\
                    &= \ppp\vvv\int_{\theta}^{\theta+2\pi} m_{\alpha,\beta,\theta}(e^{it})\;ie^{it}dt,\notag\\
                    \label{improper-cpv} &= \int_{\theta}^{\theta+2\pi} m_{\alpha,\beta,\theta}(e^{it})\;ie^{it}dt,\\
                    \label{lim-C-contour-integral} &= \int_{|z|=1}m_{\alpha,\beta,\theta}(z)\;dz.
                \end{align}
                To see that the value of the improper integral in \eqref{improper-cpv} is the same as its principal value, note that whenever an improper integral converges, its principal value converges as well (and to the same value). By convention, \eqref{improper-cpv} is evaluated as
                \begin{equation}
                    \int_\theta^{\theta + 2\pi} m_{\alpha,\beta,\theta}(e^{it})\;ie^{it}dt = \lim_{\varepsilon \to 0}\int_{\theta+\varepsilon}^{\theta + \pi} m_{\alpha,\beta,\theta}(e^{it})\;ie^{it}dt + \lim_{\varepsilon' \to 0}\int_{\theta+\pi}^{\theta + 2\pi -\varepsilon'}m_{\alpha,\beta,\theta}(e^{it})\;ie^{it}dt. \label{improper}
                \end{equation}
                It suffices to show that $g(t) = m_{\alpha,\beta,\theta}(e^{it})\;ie^{it}$ is bounded on $[\theta,\theta+2\pi]$ in order for the right hand side of \eqref{improper} to converge, and thus for the desired improper integral to converge. We first bound the real part of $\beta\log_\theta(z)$, noting that
                \begin{align}
                    \beta\log_\theta(z) &= \left[\Re(\beta)+i\Im(\beta)\right]\left[\ln|z|+i\arg_\theta(z)\right],\notag\\
                    \label{b*logz} &= [\Re(\beta)\ln|z|-\Im(\beta)\arg_\theta(z)]+i[\Re(\beta)\arg_\theta(z)+\Im(\beta)\ln|z|],
                \end{align}
                where $\arg_\theta := \Im(\log_\theta)$. Since we fix $\log_\theta(1)=0$ for every $0<\theta<2\pi$, the continuity of $\log_\theta$ on its simply connected domain implies that
                \begin{equation}\label{arg-bound-sc}
                    -2\pi < \arg_\theta(z) < 2\pi\notag
                \end{equation}
                for all $z$ in the domain and for all $\theta$. Further, continuity also implies that even as $z$ approaches the branch cut (in a limiting sense),
                \begin{equation}\label{arg-bound}
                    -2\pi \leq \arg_\theta(z) \leq 2\pi.
                \end{equation}
                Equations \eqref{b*logz} and \eqref{arg-bound} along with the assumption $\Re(\beta)>0$ give the bound
                \begin{equation}\label{re-bound}
                    \Re(\beta\log_\theta(z)) = \Re(\beta)\ln|z|-\Im(\beta)\arg_\theta(z) \leq \Re(\beta)\ln|z|+2\pi|\Im(\beta)|.
                \end{equation}
                Since $|e^z|=e^{\Re(z)}$, we can now bound 
                \begin{align}
                    |m_{\alpha,\beta,\theta}(z)| &= \frac{|e^{\beta\log_\theta(z)}|}{|z-\alpha|},\notag\\
                    \label{pre-M-bound} &= \frac{e^{\Re(\beta\log_\theta(z))}}{|z-\alpha|},\\
                    \label{M-bound} &\leq \frac{e^{\Re(\beta)\ln|z|+2\pi|\Im(\beta)|}}{||z|-|\alpha||}.
                \end{align}
                    
                For $t\in[\theta,\theta+2\pi]$, the bound \eqref{M-bound} immediately gives
                \begin{equation}\label{cpv-bound}
                    |m_{\alpha,\beta,\theta}(e^{it})\;ie^{it}| \leq \frac{e^{\Re(\beta)\ln|e^{it}|+2\pi|\Im(\beta)|}}{||e^{it}|-|\alpha||} = \frac{e^{2\pi|\Im(\beta)|}}{|1-|\alpha||}.
                \end{equation}
                Thus both limits on the right hand side of \eqref{improper} converge, and hence the equality in \eqref{improper-cpv} is justified.

            %\subsubsection{$D_{\eps,\theta,\rho}$}
                Now we show that the portion of the integral over the contour $D_{\eps,\theta,\rho}$ approaches 0 as $\eps \rightarrow 0^+, \rho \rightarrow 0^+$.
                
                Using \eqref{M-bound} and applying an ML-bound to \eqref{D-param} yields
                \begin{align}
                    \left|\int_{D_{\eps,\theta,\rho}} m_{\alpha,\beta,\theta}\right| &= \left|\int_{\theta +2\pi - \varepsilon}^{\theta+\eps} m_{\alpha,\beta,\theta}(\rho e^{it})\;i\rho e^{it}dt\right|,\notag\\
                    &\leq \rho\frac{e^{\Re(\beta)\ln|\rho|+2\pi|\Im(\beta)|}}{|\rho-|\alpha||}(2\pi-2\eps),\notag\\
                    \label{ml-bound} &\leq 2\pi e^{2\pi|\Im(\beta)|}\frac{\rho^{\Re(\beta)}}{|\frac{|\alpha|}{\rho}-1|}.
                \end{align}
                Since $|\alpha|>0$, \eqref{ml-bound} gives
                \begin{align}
                    \left|\lim_{\rho\to0^+}\lim_{\eps\to0^+}\int_{D_{\eps,\theta}}m_{\alpha,\beta,\theta} \right| &\leq \lim_{\rho\to0^+}\lim_{\eps\to0^+}\left[2\pi e^{2\pi|\Im(\beta)|}\frac{\rho^{\Re(\beta)}}{|\frac{|\alpha |}{\rho}-1|}\right],\notag\\
                    &= \lim_{\rho\to0^+}\left[2\pi e^{2\pi|\Im(\beta)|}\frac{\rho^{\Re(\beta)}}{|\frac{|\alpha|}{\rho}-1|}\right],\notag\\
                    \label{lim-D-contour-integral} &= 0.
                \end{align}

            %\subsubsection{$L_{\eps,\theta,\rho}$}
                We now consider the limiting value of the integral along $L_{\eps,\theta,\rho}$. The core difficulty of this part of the contour integral is in evaluating
                \begin{equation}\label{eps-lim-L}
                    \lim_{\eps\to0^+}\int_\rho^1 m_{\alpha,\beta,\theta}(te^{i(\theta + \varepsilon)})e^{i(\theta + \varepsilon)}\;dt.
                \end{equation}
                The strategy is to use Lebesgue's Dominated Convergence Theorem.
                Take the family of functions defined on $[0,1]$
                \begin{equation}\label{fam-L}
                    \fff_L = \left\{f_\eps(t) = m_{\alpha,\beta,\theta}(te^{i(\theta + \eps)})e^{i(\theta + \eps)} \;\Big|\; 0<\eps<\pi\right\}.
                \end{equation}
                Since $m_{\alpha,\beta,\theta}$ is continuous on its simply connected domain except at the point $\alpha$ (which has measure 0), and because $te^{i(\theta + \eps)}$ is a continuous function in the positive real variable $t$, $\fff_L$ is a family of almost everywhere continuous functions. Luzin's Criterion implies that functions in $\fff_L$ are Lebesgue measurable \cite{efimov}.
                
                For any $f_\eps\in\fff_L$, we have
                \begin{align}
                    \lim_{\eps\to0^+}f_\eps(t) &= \lim_{\eps\to0^+}m_{\alpha,\beta,\theta}(te^{i(\theta + \varepsilon)})e^{i(\theta + \varepsilon)},\notag\\
                    &= \lim_{\eps\to0^+}\frac{e^{\beta\log_\theta(te^{i(\theta + \varepsilon)})}}{te^{i(\theta + \varepsilon)}-\alpha}e^{i(\theta + \varepsilon)},\notag\\
                    &= \lim_{\eps\to0^+}\frac{e^{\beta(\ln{t}+i(\theta - 2\pi + \varepsilon))}}{te^{i(\theta + \varepsilon)}-\alpha}e^{i(\theta + \varepsilon)},\notag\\
                    \label{fL-limeps0} &= \frac{e^{\beta(\ln{t}+i(\theta-2\pi))}}{te^{i\theta}-\alpha}e^{i\theta} =: g_\theta(t).
                \end{align}
                (Note $t\in[\rho,1]\subseteq(0,\infty)$ above, and since $\Arg(\alpha)\neq\theta$ we have that $te^{i\theta}-\alpha\neq0$ in the limit. The limit is evaluated using this fact along with the continuity of the exponential.)\\
                Equivalently, this means that for any sequence $\eps_n\to0^+$, $f_{\eps_n}$ converges pointwise to $g_\theta$ as $n\to\infty$.\\
                
                Fix $r$ with $0<r<\min_{\arg(\alpha)}|\arg(\alpha)-\theta|$, and consider
                \begin{equation}\label{fam-L^*}
                    \fff_L^* = \left\{f_\eps(t) = m_{\alpha,\beta,\theta}(te^{i(\theta + \eps)})e^{i(\theta + \eps)} \;\Big|\;0<\eps<r\right\}.
                \end{equation}
                Since $\fff_L^*\subseteq\fff_L$, the statement about pointwise convergence for $\fff_L$ also holds for $\fff_L^*$. It is worth mentioning that $\alpha$ is outside of the sector between $\theta-r$ and $\theta+r$, which implies there are no issues with boundedness for functions in $\fff_L^*$.
                
                Let $\delta:=\min_{\arg(\alpha)}\{|\arg(\alpha)-(\theta+\eps)|\}$. That is, $\delta$ gives the minimum difference in angle between $\theta+\eps$ and the vector from the origin out to $\alpha$.
                
                If $\delta\geq\tfrac{\pi}{2}$, simple geometry gives that $\alpha$ is at least a distance of $|\alpha|$ away from the segment $te^{i(\theta+\eps)}$ for $t\in[\rho,1]$. To see this, consider Figure \ref{delta-geq-pi/2} 
                \begin{figure}[h!]
                    \begin{center}
                        %\begin{tikzpicture}
                            %\filldraw[gray] (0,0) circle (3pt);
                            %\draw[->, gray] (0,0) -- (-4,0) node[left]{$\theta$};
                            %\draw[->, thick, dotted] (-160:0) -- (-160:4) node[left]{$\theta+\eps$};
                            %\draw[thick] (-160:1) -- (-160:3.5);
                            %\draw[dashed] (110:2)--(110:-4);
                            %\draw[blue] (-160:0.5) -- (-205:0.7) -- (-250:0.5);
                            %\draw[red] (-160:0.4) arc (-160:-55:0.4) node[midway, below, yshift=1]{$\delta$};
                            %\draw (-55:0) -- node[midway, xshift=7, yshift=7]{$|\alpha|$} (-55:3) -- (-160:1);
                            %\filldraw[red] (-55:3) circle (3pt) node[right, xshift=2]{$\alpha$};
                        %\end{tikzpicture}
                        \includegraphics{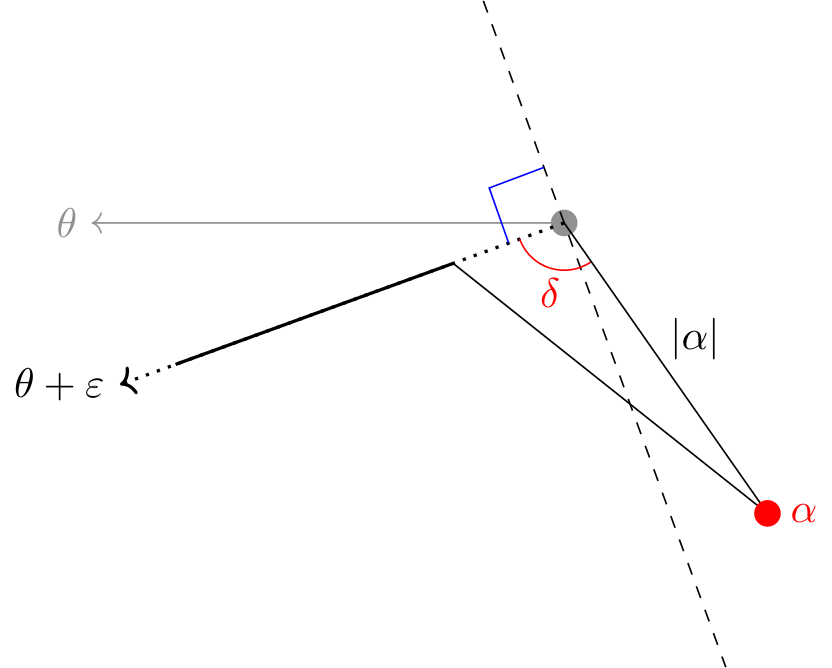}
                        \caption{Illustration of the case $\delta \geq \frac{\pi}{2}$.}
                        \label{delta-geq-pi/2}
                    \end{center}
                \end{figure}
                and note that the side of the triangle opposite the angle of size $\delta$ is the longest side of the triangle (since $\delta$ is either right or obtuse). Thus the shortest distance $d$ from $\alpha$ to the line segment $L_{\epsilon,\theta,\rho}$ is bounded below: 
                \begin{equation}\label{lb-alpha-far}
                    d > |\alpha|.
                \end{equation}
                
                If instead $\delta<\tfrac{\pi}{2}$, then $\alpha$ is at least a distance of $|\alpha|\sin(\delta)$ away from the segment $te^{i(\theta+\eps)}$ for $t\in[\rho,1]$. To see this consider the similar picture in Figure \ref{delta-less-pi/2} 
                \begin{figure}[h!]
                    \begin{center}
                        %\begin{tikzpicture}
                            %\filldraw[gray] (0,0) circle (3pt);
                            %\draw[->, gray] (0,0) -- (-4,0) node[left]{$\theta$};
                            %\draw[->, thick, dotted] (-160:0) -- (-160:4) node[left]{$\theta+\eps$};
                            %\draw[->, dash dot] (-155:0) -- (-155:4) node[left]{$\theta+r$};
                            %\draw[->, dash dot] (155:0) -- (155:4) node[left]{$\theta-r$};
                            %\draw[thick] (-160:1) -- (-160:3.5);
                            %\draw[dashed, decoration={markings, mark=at position 0.55 with {\arrow{>}}}, postaction={decorate}] (110:2)--(110:-4);
                            %\draw[blue] (-160:0.5) -- (-205:0.7) -- (-250:0.5);
                            %\draw[red] (-160:0.4) arc (-160:-100:0.4) node[midway, below, xshift=-3, yshift=2]{$\delta$};
                            %\draw (-100:0) -- node[midway, right]{$|\alpha|$} (-100:3);
                            %\draw[decoration={markings, mark=at position 0.55 with {\arrow{<}}}, postaction={decorate}]  (-100:3) -- (-160:3*cos(60);));
                            %\filldraw[red] (-100:3) circle (3pt) node[below, yshift=-2]{$\alpha$};
                        %\end{tikzpicture}
                       \includegraphics{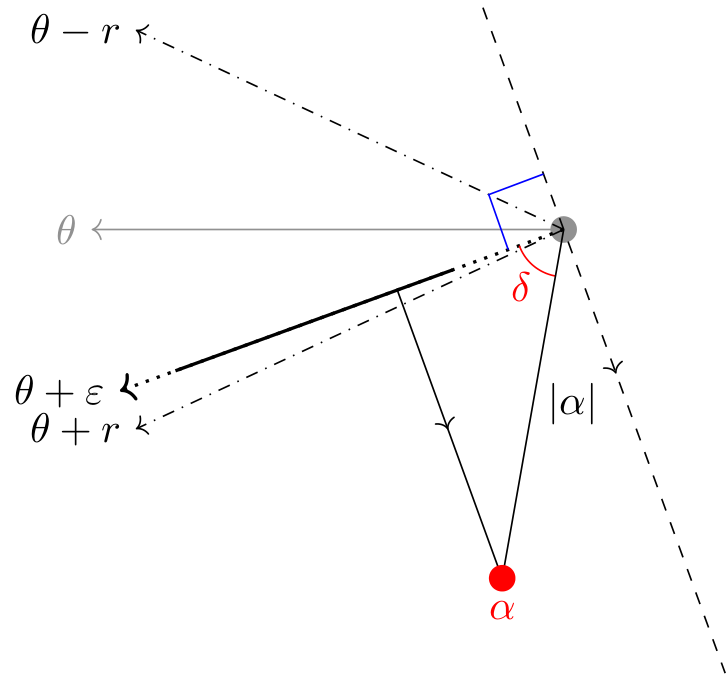}
                        \caption{Illustration of the case $\delta < \frac{\pi}{2}$.}
                        \label{delta-less-pi/2}
                    \end{center}
                \end{figure}
                and note that the altitude dropped from $\alpha$ to the line containing $L_{\eps,\theta,\rho}$ is precisely of length $|\alpha|\sin(\delta)$ (although the distance will be greater if $|\alpha|$ is so small or so large that the altitude dropped onto the line does not strike within the segment parameterized by $t\in[\rho,1]$).\\
                Now in our consideration of $F_L^*$, we have $\eps<r<\delta$ and so
                \begin{equation}
                    |\alpha|\sin(\delta) > |\alpha|\sin\left(\min_{\arg(\alpha)}\big\{|\arg(\alpha)-(\theta\pm r)|\big\}\right) > k > 0,\notag
                \end{equation}
                for a fixed constant $k$ only dependent on $r$, $\theta$, and $\alpha$.
                Thus in this case as well, the shortest distance $d$ from $\alpha$ to the line segment $L_{\epsilon,\theta,\rho}$ is bounded below:
                \begin{equation}\label{lb-alpha-close}
                    d > k.
                \end{equation}
                ~\\
                
                Consequently, for $K:=\min\{\tfrac{1}{|\alpha|},\tfrac{1}{k}\}$, we have that every function $f_\eps\in\fff_L^*$ has for all $t\in[\rho,1]$ that
                \begin{align}
                    |f_\eps(t)| &= |m_{\alpha,\beta,\theta}(te^{i(\theta + \eps)})e^{i(\theta + \eps)}|,\notag\\
                    &= |m_{\alpha,\beta,\theta}(te^{i(\theta + \eps)})|,\notag\\
                    &= \frac{e^{\Re(\beta\log_\theta(te^{i(\theta + \eps)}))}}{|te^{i(\theta + \eps)}-\alpha|},\\
                    &\leq Ke^{\Re(\beta)\ln|te^{i(\theta + \eps)}|+2\pi|\Im(\beta)|},\notag\\
                    \label{F_L^*-bound} &\leq Ke^{\Re(\beta)|t|+2\pi|\Im(\beta)|} =: h_r(t).
                \end{align}
                (The third equality holds by \eqref{pre-M-bound}; the first inequality holds by \eqref{re-bound} and the reasoning which led to \eqref{lb-alpha-far} and \eqref{lb-alpha-close}; the final inequality holds since $|t|>\ln|t|$ for all $t\in\rr$ and the because exponential is strictly increasing on $\rr$.)\\
                Clearly this $h_r$ is integrable on $[\rho,1]$ for all $\rho>0$, since it is simply a scaled exponential.\\
                
                Consider any arbitrary sequence $(\eps_n)\to0^+$ with $\eps_n < r$, and define a sequence of functions $(f_{\eps_n})$; note $f_{\eps_n}\in\fff_L^*$ for all $n$. From \eqref{fL-limeps0} we have $(f_{\eps_n})\to g_\theta$ pointwise, and $|f_{\eps_n}(t)|\leq h_r(t)$ for all $n$ and for all $t\in[0,1]$, as shown in \eqref{F_L^*-bound}. Therefore Lebesgue's Dominated Convergence Theorem implies that
                \begin{equation}\label{DT-single-L^*}
                    \lim_{n\to\infty}\int_\rho^1 f_{\eps_n}(t)\;dt = \int_\rho^1 g_\theta(t)\;dt.
                \end{equation}
                for all $\rho>0$.\\
                Dispensing with the condition $\eps_n < r$, it is still true for any arbitrary sequence $(\eps_n)\to0^+$ 
                that 
                \begin{equation}\label{DT-single-L}
                    \lim_{n\to\infty}\int_\rho^1 f_{\eps_n}(t)\;dt = \int_\rho^1 g_\theta(t)\;dt
                \end{equation}
                since $(\eps_n)\to0^+$ has a tail which is completely bounded above by $r$, and thus convergence of the tail shown in \eqref{DT-single-L^*} implies convergence of the whole sequence. Since \eqref{DT-single-L} holds for arbitrary $(\eps_n)$, this implies
                \begin{equation}\label{DT-L}
                    \lim_{\eps\to0^+}\int_\rho^1 f_\eps(t)\;dt = \int_\rho^1 g_\theta(t)\;dt.
                \end{equation}
                for $f_\eps \in \fff_L$.\\
                
               Hence we evaluate \eqref{eps-lim-L} and find
                \begin{align}
                    \lim_{\eps\to0^+}\int_\rho^1 m_{\alpha,\beta,\theta}(te^{i(\theta + \varepsilon)})e^{i(\theta + \varepsilon)}\;dt &= \int_\rho^1 \lim_{\eps\to0^+}m_{\alpha,\beta,\theta}(te^{i(\theta + \varepsilon)})e^{i(\theta + \varepsilon)}dt,\notag\\
                    &= \int_\rho^1 \frac{e^{\beta(\ln{t}+i(\theta-2\pi))}}{te^{i\theta}-\alpha}e^{i\theta}\;dt,\notag\\
                    \label{eps-lim-L-eval} &= e^{i\beta(\theta-2\pi)}\int_\rho^1 \frac{e^{\beta\ln{t}}}{t-\alpha e^{-i\theta}}\;dt.
                \end{align}
                But $g_\theta$ is continuous for $t\in(0,1]$ and bounded as $t\to0$. Therefore, allowing improper integrals, and drawing from equations \eqref{L-param} and \eqref{eps-lim-L-eval} it is straightforward to compute
                \begin{align}
                    \lim_{\rho\to0^+}\lim_{\eps\to0^+}\int_{L_{\eps,\theta,\rho}} m_{\alpha,\beta,\theta} &= \lim_{\rho\to0^+}\lim_{\eps\to0^+}\int_\rho^1 m_{\alpha,\beta,\theta}(te^{i(\theta + \varepsilon)})e^{i(\theta + \varepsilon)}\;dt,\notag\\
                    &= \lim_{\rho\to0^+}\left[e^{i\beta(\theta-2\pi)}\int_\rho^1 \frac{e^{\beta\ln{t}}}{t-\alpha e^{-i\theta}}\;dt\right],\notag\\
                    \label{lim-L-contour-integral} &= e^{i\beta(\theta-2\pi)}\int_0^1 \frac{e^{\beta\ln{t}}}{t-\alpha e^{-i\theta}}\;dt.
                \end{align}

            %\subsubsection{$M_{\eps,\theta,\rho}$}
                Finally we take limits in the last integral on the right hand side of \eqref{lim-contour-integral} along $M_{\eps,\theta,\rho}$. Similarly the difficulty in this case is evaluating
                \begin{equation}\label{eps-lim-M}
                    \lim_{\eps\to0^+}\int_1^\rho m_{\alpha,\beta,\theta}(te^{i(\theta + 2\pi - \varepsilon)})e^{i(\theta + 2\pi - \varepsilon)}\;dt
                \end{equation}
                using Lebesgue's Dominated Convergence Theorem. Analogous steps as those used for $L_{\eps,\theta,\rho}$ can be applied to the $M_{\eps,\theta,\rho}$ case to show that 
                \begin{align}
                    \lim_{\eps\to0^+}\int_1^\rho m_{\alpha,\beta,\theta}(te^{i(\theta + 2\pi - \varepsilon)})e^{i(\theta + 2\pi - \varepsilon)}\;dt &= -\int_\rho^1 \lim_{\eps\to0^+}m_{\alpha,\beta,\theta}(te^{i(\theta + 2\pi - \varepsilon)})e^{i(\theta + 2\pi - \varepsilon)}\;dt,\notag\\
                    &= -\int_\rho^1 \frac{e^{\beta(\ln(t)+i\theta)}}{te^{i(\theta+2\pi)}-\alpha}e^{i(\theta+2\pi)}\;dt,\notag\\
                    \label{eps-lim-M-eval} &= -e^{i\beta\theta}\int_\rho^1 \frac{e^{\beta\ln{t}}}{t-\alpha e^{-i\theta}}\;dt.
                \end{align}
                Just as before the integrand is bounded on $[0,1]$. Using \eqref{eps-lim-M-eval} there is no issue writing
                \begin{align}
                    \lim_{\rho\to0^+}\lim_{\eps\to0^+}\int_{M_{\eps,\theta,\rho}} m_{\alpha,\beta,\theta} &= \lim_{\rho\to0^+}\lim_{\eps\to0^+}\int_1^\rho m_{\alpha,\beta,\theta}(te^{i(\theta + 2\pi - \varepsilon)})e^{i(\theta + 2\pi - \varepsilon)}\;dt,\notag\\
                    &= \lim_{\rho\to0^+}\left[-e^{i\beta\theta}\int_\rho^1 \frac{e^{\beta\ln{t}}}{t-\alpha e^{-i\theta}}\;dt\right],\notag\\
                    \label{lim-M-contour-integral} &= -e^{i\beta\theta}\int_0^1 \frac{e^{\beta\ln{t}}}{t-\alpha e^{-i\theta}}\;dt.
                \end{align}

    \section{The Hypergeometric Function Connection} \label{hypergeo-section}
        
        \subsection{A Core Integral} \label{core}
            In order to fully evaluate \eqref{main-integral} using the contour method outlined in \S\ref{contour-1}, the following integral from Lemma \ref{contour-lemma} must be evaluated: 
            \begin{equation}\label{pre-cov-integral}
                \int_0^1\frac{e^{\beta\ln{t}}}{t-\alpha e^{-i\theta}}\;dt,
            \end{equation}
            which exists for $\Re(\beta)>-1$.
            The integral in \eqref{pre-cov-integral} is in fact an improper integral and can be written
            \begin{equation}
                \lim_{\rho\to0}\int_\rho^1\frac{e^{\beta\ln{t}}}{t-\alpha e^{-i\theta}}\;dt.\notag
            \end{equation}
            Algebraic manipulations give
            \begin{align}
                \frac{e^{\beta\ln{t}}}{t-\alpha e^{-i\theta}} &= \frac{e^{\ln{t}}}{t-\alpha e^{-i\theta}}e^{(\beta-1)\ln{t}},\notag\\
                &= \frac{t-\alpha e^{-i\theta}+\alpha e^{-i\theta}}{t-\alpha e^{-i\theta}}e^{(\beta-1)\ln{t}},\notag\\
                \label{alg-manip} &= \left(1+\frac{\alpha e^{-i\theta}}{t-\alpha e^{-i\theta}}\right)e^{(\beta-1)\ln{t}}.
            \end{align}
            Integrating first over the interval $[\rho,1]$ and using \eqref{alg-manip} yields
            \begin{equation}\label{split-int}
                \int_\rho^1 \frac{e^{\beta\ln{t}}}{t-\alpha e^{-i\theta}}\;dt = \int_\rho^1 e^{(\beta-1)\ln{t}}\;dt + \alpha e^{-i\theta}\int_\rho^1 \frac{e^{(\beta-1)\ln{t}}}{t-\alpha e^{-i\theta}}\;dt.
            \end{equation}
            Notice that the function $e^{(\beta-1)\log(t)}$ is the derivative of $\frac{1}{\beta}e^{\beta\log(t)},$ which is analytic on $[\rho,1]$. Thus 
            %prop 4.12 in Bak and Newman
            \begin{equation}\label{split-int-first-term}
                \int_\rho^1 e^{(\beta-1)\ln{t}}\;dt = \frac{1}{\beta}e^{\beta\ln(1)} - \frac{1}{\beta}e^{\beta\ln(\rho)} = \frac{1}{\beta} - \frac{1}{\beta}e^{\beta\ln(\rho)}.
            \end{equation}
            Substituting \eqref{split-int-first-term} into \eqref{split-int} and taking limits gives
            \begin{align}
                \lim_{\rho\to0}\int_\rho^1 \frac{e^{\beta\ln{t}}}{t-\alpha e^{-i\theta}}\;dt &= \lim_{\rho\to0}\left[\frac{1}{\beta} - \frac{1}{\beta}e^{\beta\ln(\rho)}\right] + \alpha e^{-i\theta}\lim_{\rho\to0}\int_\rho^1 \frac{e^{(\beta-1)\ln{t}}}{t-\alpha e^{-i\theta}}\;dt,\notag\\
                \label{close-to-2F1} \int_0^1 \frac{e^{\beta\ln{t}}}{t-\alpha e^{-i\theta}}\;dt &= \frac{1}{\beta} + \alpha e^{-i\theta}\int_0^1 \frac{e^{(\beta-1)\ln{t}}}{t-\alpha e^{-i\theta}}\;dt;
            \end{align}
            the integral on the right hand side of \eqref{close-to-2F1} exists for $\Re(\beta)>0$.
            Again the convergence of improper integrals follows from the boundedness of the integrands. Moving the constant inside the integral in \eqref{close-to-2F1} gives
            \begin{equation}\label{basically-2F1}
                \int_0^1 \frac{e^{\beta\ln{t}}}{t-\alpha e^{-i\theta}}\;dt = \frac{1}{\beta} - \int_0^1 \frac{e^{(\beta-1)\ln{t}}}{1-\left(\frac{1}{\alpha}e^{i\theta}\right)t}\;dt.
            \end{equation}
            Therefore finding a solution to \eqref{main-integral} using the contour integration method necessitates working with the following ``core integral'' for $z=\frac{1}{\alpha}e^{i\theta}$:
            \begin{equation}\label{core-integral}
                \int_0^1 t^{\beta-1}(1-zt)^{-1}\;dt.
            \end{equation}
            The choice to write $t^{\beta-1}$ rather than $e^{(\beta-1)\ln{t}}$ in \eqref{core-integral} is intentional, since generality is not lost when any branch $\log_\theta$ for $\theta\not\equiv0$ is used to define this complex power of $t\in[0,1]$.
            
        \subsection{Definition \& Relevant Identities}
            We investigate the integral in \eqref{core-integral} by making use of the well-studied hypergeometric function $_2 F_1(a,b,c;z)$. For $|z|<1$, this function is defined as the infinite series
            \begin{equation}\label{hypgeo-def}
                _2F_1(a,b,c;z) = \sum_{n=0}^\infty \frac{(a)_n (b)_n}{(c)_n n!} z^n, \ \ \ c \in \cc \setminus \zz_{\le0}
            \end{equation}
            where $(x)_n = \frac{\Gamma(x+n)}{\Gamma(x)}$ is the rising Pochhammer symbol.
            
            The hypergeometric series generalizes the geometric series, and is prominent in the study of linear differential equations with three regular singular points. The hypergeometric function is notably a solution to the hypergeometric equation, discussed in \S\ref{diffeq}.

            A comprehensive collection of identities involving $_2 F_1$ can be found in \cite{Erde53}. The most notable for our purposes is the following:
            
            For $|z|<1$ and $\Re(c)>\Re(b)>0$,
            \begin{equation}\label{euler-type-integral}
                _2F_1(a,b,c;z) = \frac{\Gamma(c)}{\Gamma(b)\Gamma(c-b)}\int_0^1 t^{b-1}(1-t)^{c-b-1}(1-tz)^{-a}\,dt.
            \end{equation}
            
            Letting $a=1$, $b=\beta$, and $c=1+\beta$ under the conditions for \eqref{euler-type-integral} gives
            \begin{align}
                _2F_1(1,\beta,1+\beta;z) &= \frac{\Gamma(1+\beta)}{\Gamma(\beta)\Gamma(1)}\int_0^1 t^{\beta-1}(1-t)^{0}(1-tz)^{-1}\,dt,\notag\\
                \label{key-2F1-id} &= \beta\int_0^1 t^{\beta-1}(1-zt)^{-1}\,dt,
            \end{align}
            such that the integral above is exactly the integral in \eqref{core-integral}, only scaled.

        \subsection{Final Steps of the Contour Method}
            We conclude the contour method for $|\alpha|>1$ by proving the following statement, making use of the hypergeometric identity \eqref{key-2F1-id}.
            \begin{prop}\label{contour-prop}
                When $|\alpha|>1$, $\Arg(\alpha)\neq\theta$, and $\Re(\beta)>0$,
                \begin{equation}\label{alpha>1-case}
                    \int_{\del\dd}m_{\alpha,\beta,\theta} = e^{i\beta\theta}(1-e^{-2\pi i\beta})\frac{1}{\beta}\left[1-\,_2F_1(1,\beta;1+\beta;\alpha^{-1}e^{i\theta})\right].
                \end{equation}
                \begin{proof}
                    Since $|\alpha|>1$, the last argument in the hypergeometric function satisfies
                    \begin{equation}
                         \left|\frac{1}{\alpha}e^{i\theta}\right| = \frac{1}{|\alpha|} < 1.\notag
                    \end{equation}
                    Under the assumption $\Re(\beta)>0$, one can apply the identity \eqref{key-2F1-id} and find
                    \begin{equation}
                        \int_0^1 \frac{t^{\beta-1}}{1-\left(\alpha^{-1}e^{i\theta}\right)t}\;dt = \frac{1}{\beta}\, _2F_1(1,\beta;1+\beta;\alpha^{-1}e^{i\theta}).
                    \end{equation}
                    With this expression for the core integral,  an application of Lemma \ref{contour-lemma} and \eqref{basically-2F1} completes the proof.
                \end{proof}
            \end{prop}
            
            The case $0<|\alpha|<1$ cannot be approached in the same manner. While the initial steps in the contour method still hold, the integral identity from \eqref{key-2F1-id} does not apply since the last argument in the hypergeometric function now satisfies
            \begin{equation}
                \left|\frac{1}{\alpha}e^{i\theta}\right| = \frac{1}{|\alpha|} > 1;\notag
            \end{equation}
            which is outside the domain of \eqref{euler-type-integral}.

    \section{Series Method} \label{series-method-section}
        Fortunately there exist methods outside of contour integration which allow us to express \eqref{main-integral} in terms of the hypergeometric function in all cases. Rather than dealing with integral identities of the hypergeometric function, one can work with series to produce a term of the form \eqref{hypgeo-def}.

        \subsection{Proof of the main result.}
            Consider first $|\alpha|>1$. Recall from \eqref{improper-cpv} and \eqref{lim-C-contour-integral} that
            \begin{align}
                \int_{\del\dd} m_{\alpha,\beta,\theta} &=  \int_{|z|=1}\frac{e^{\beta\log_\theta(z)}}{z-\alpha}\,dz,\notag\\
                \label{alph>1-pv-int} &= \lim_{\eps\to0^+}\int_{\theta-2\pi+\eps}^{\theta-\eps} m_{\alpha,\beta,\theta}(e^{it})ie^{it}\;dt.
            \end{align}
            For $t \in (\theta-2\pi, \theta)$, $\log_\theta(e^{it})=it$, so one can rewrite the integrand as
            \begin{align}
                m_{\alpha,\beta,\theta}(e^{it})ie^{it} &= \frac{e^{\beta(\log_\theta(e^{it}))}}{e^{it}-\alpha}ie^{\log_\theta(e^{it})},\notag\\
                &= i\frac{e^{(\beta+1)(\log_\theta(e^{it}))}}{e^{it}-\alpha},\notag\\
                &= i\frac{e^{i(\beta+1)t}}{e^{it}-\alpha},\notag\\
                &= -\frac{ie^{i(\beta+1)t}}{\alpha}\cdot\frac{1}{1-\frac{1}{\alpha}e^{it}},\notag\\
                \label{alph>1-frac-to-geo} &= -\frac{ie^{i(\beta+1)t}}{\alpha}\sum_{k=0}^\infty \alpha^{-k}e^{ikt};
            \end{align}
            where \eqref{alph>1-frac-to-geo} follows by rewriting in terms of a convergent geometric series. Pulling the factor of $e^{it}$ inside the series yields   
            \begin{align}
                m_{\alpha,\beta,\theta}(e^{it})ie^{it} &= -ie^{i\beta t}\sum_{k=0}^\infty \alpha^{-(k+1)}e^{i(k+1)t},\notag\\
                &= -ie^{i\beta t}\sum_{k=1}^\infty \alpha^{-k}e^{ikt},\notag\\
                \label{alph>1-geo-series} &= -i\sum_{k=1}^\infty \alpha^{-k}e^{i(\beta+k)t}.
            \end{align}
            For a fixed $|\alpha|>1$ we have that $|\alpha|^{-1}<1$, so 
            \begin{equation}
                \left|\sum_{k=1}^\infty \alpha^{-k}e^{ikt}\right| \leq \sum_{k=1}^\infty |\alpha^{-k}e^{ikt}| = \sum_{k=1}^\infty |\alpha|^{-k} =: K_\alpha < \infty.\notag
            \end{equation}
            We define a sequence of functions $(f_n)$, where $f_n:[\theta-2\pi,\theta]\to\cc$ are given by 
            \begin{equation}
                f_n(t) := \sum_{k=1}^n \alpha^{-k}e^{i(\beta+k)t} = e^{i\beta t}\sum_{k=1}^n \alpha^{-k}e^{ikt}.\notag
            \end{equation}
            This sequence is uniformly bounded, since
            \begin{align}
                |f_n(t)| &= \left|e^{i\beta t}\sum_{k=1}^n \alpha^{-k}e^{ikt}\right|,\notag\\
                &\leq K_\alpha |e^{i\beta t}|,\notag\\
                &= K_\alpha e^{-\Im(\beta)t} =: g_\alpha(t).\notag
            \end{align}
            Note that $g_\alpha$ is integrable on $[\theta-2\pi+\eps,\theta-\eps]$ since it is simply a scaled exponential.
            It is clear that each $f_n$ is continuous as a finite sum of analytic functions, so again by Luzin's Criterion the functions are measurable \cite{efimov}.
            Additionally, their pointwise limit is the expression in \eqref{alph>1-geo-series}. Thus Lebesgue's Dominated Convergence Theorem implies 
            \begin{align}
                \int_{\theta-2\pi+\eps}^{\theta-\eps}\sum_{k=1}^\infty\alpha^{-k}e^{i(\beta+k)t}\;dt &= \lim_{n\to\infty} \int_{\theta-2\pi+\eps}^{\theta-\eps} \sum_{k=1}^n \alpha^{-k}e^{i(\beta+k)t}\;dt,\notag\\
                &= \lim_{n\to\infty}\sum_{k=1}^n \int_{\theta-2\pi+\eps}^{\theta-\eps}\alpha^{-k}e^{i(\beta+k)t}\;dt,\notag\\
                \label{alph>1-int-sum-swap} &= \sum_{k=1}^\infty \int_{\theta-2\pi+\eps}^{\theta-\eps}\alpha^{-k}e^{i(\beta+k)t}\;dt,
            \end{align}
            where the second equality holds since it is merely the interchange of an integral and finite sum.\\
                
            Using \eqref{alph>1-int-sum-swap} along with \eqref{alph>1-geo-series} yields
            \begin{align}
                \int_{\theta-2\pi+\eps}^{\theta-\eps} m_{\alpha,\beta,\theta}(e^{it})ie^{it}\;dt &= -i\int_{\theta-2\pi+\eps}^{\theta-\eps}\sum_{k=1}^\infty \alpha^{-k}e^{i(\beta+k)t}\;dt,\notag\\
                &= -i\sum_{k=1}^\infty \int_{\theta-2\pi+\eps}^{\theta-\eps}\alpha^{-k}e^{i(\beta+k)t}\;dt,\notag\\
                \label{alph>1-swapped-int-sum} &= -i\sum_{k=1}^\infty \alpha^{-k} \int_{\theta-2\pi+\eps}^{\theta-\eps}e^{i(\beta+k)t}\;dt.
            \end{align}
                
            An individual summand of \eqref{alph>1-swapped-int-sum} consists of an $\alpha^{-k}$ term multiplied by an integral.
            The integrand, $e^{i(\beta+k)t}$, is an entire function of $t$ and thus is bounded on the compact set $t\in[\theta-2\pi,\theta]$ by some $M$. Note that this bound $M$ can be chosen independent of $k$ since
            \begin{equation}
                e^{-\Im(\beta)t} = |e^{i(\beta+k)t}| < M \qquad\forall t\in[\theta-2\pi,\theta].\notag
            \end{equation}
                
            The length of the curve being integrated over is at most
            \begin{equation}
                (\theta-\eps)-(\theta-2\pi+\eps) = 2\pi-2\eps < 2\pi =: L,\notag
            \end{equation}
            where $L$ does not depend on $\eps$.\\
                
           Because the integrand is entire it is must be continuous on the path of integration, and so the $ML$-bound gives that 
            \begin{equation}
                \left|\int_{\theta-2\pi+\eps}^{\theta-\eps}e^{i(\beta+k)t}\;dt\right| \leq ML,\notag
            \end{equation}
            where $M$ and $L$ are given above and independent of $\eps$ and $k$. Thus each term of the series in \eqref{alph>1-swapped-int-sum} is bounded in modulus by $ML|\alpha|^{-k}$, so that
            \begin{equation}\label{M-test}
                \left|\sum_{k=1}^\infty \alpha^{-k} \int_{\theta-2\pi+\eps}^{\theta-\eps}e^{i(\beta+k)t}\;dt\right| \leq \sum_{k=1}^\infty ML|\alpha|^{-k}.
            \end{equation}
            Since $|\alpha|>1$, the right hand side of \eqref{M-test} converges, and the Weierstrass $M$-test implies that the series in \eqref{alph>1-swapped-int-sum} is uniformly convergent. Substituting the expression from \eqref{alph>1-swapped-int-sum} back into \eqref{alph>1-pv-int} allows the exchange of limit and infinite sum in \eqref{lim-infsum-xchange} to find that
            \begin{align}
                \int_{\del\dd} m_{\alpha,\beta,\theta} &= \lim_{\eps\to0^+}\left[-i\sum_{k=1}^\infty \alpha^{-k} \int_{\theta-2\pi+\eps}^{\theta-\eps}e^{i(\beta+k)t}\;dt\right],\notag\\
                \label{lim-infsum-xchange} &= -i\sum_{k=1}^\infty \alpha^{-k} \lim_{\eps\to0^+}\int_{\theta-2\pi+\eps}^{\theta-\eps}e^{i(\beta+k)t}\;dt,\\
                \label{alph>1-no-eps} &= -i\sum_{k=1}^\infty \alpha^{-k}\int_{\theta-2\pi}^{\theta}e^{i(\beta+k)t}\;dt.
            \end{align}
            The integrand in \eqref{alph>1-no-eps} is entire, and it has an antiderivative $\frac{e^{i(\beta+k)t}}{i(\beta+k)}$ when $\beta+k\neq0$; this antiderivative is also entire. This fact not only ensures the equality between \eqref{lim-infsum-xchange} and \eqref{alph>1-no-eps}, but it also allows the use of the Complex Fundamental Theorem of Calculus to conclude that for $\beta\notin\zz_{<0}$,
            \begin{equation}\label{alph>1-int-in-sum}
                \int_{\theta-2\pi}^{\theta}e^{i(\beta+k)t}\;dt = \left[\frac{e^{i(\beta+k)t}}{i(\beta+k)}\right]_{\theta-2\pi}^{\theta} = \frac{e^{i(\beta+k)\theta}}{i(\beta+k)}\left(1-e^{-2\pi i\beta}\right).
            \end{equation}
            Using definition \eqref{hypgeo-def} as well as our intermediates \eqref{alph>1-no-eps} and \eqref{alph>1-int-in-sum} we find
            \begin{align}
                \int_{\del\dd} m_{\alpha,\beta,\theta} &= -i\sum_{k=1}^\infty \alpha^{-k}\frac{e^{i(\beta+k)\theta}}{i(\beta+k)}\left(1-e^{-2\pi i\beta}\right),\notag\\
                &= -e^{i\beta\theta}\left(1-e^{-2\pi i\beta}\right)\sum_{k=1}^\infty \alpha^{-k}\frac{e^{ik\theta}}{\beta+k},\notag\\
                &= -e^{i\beta\theta}\left(1-e^{-2\pi i\beta}\right)\frac{1}{\beta}\left[\left(\sum_{k=0}^\infty \frac{\beta}{\beta+k}(\alpha^{-1}e^{i\theta})^k\right)-1\right],\notag\\
                &= -e^{i\beta\theta}\left(1-e^{-2\pi i\beta}\right)\frac{1}{\beta}\left[\left(\sum_{k=0}^\infty \frac{(1)_k (\beta)_k}{(1+\beta)_k k!} (\alpha^{-1}e^{i\theta})^k\right)-1\right],\notag\\
                &= -e^{i\beta\theta}\left(1-e^{-2\pi i\beta}\right)\frac{1}{\beta}\left[\,_2F_1(1,\beta;1+\beta;\alpha^{-1}e^{i\theta})-1\right],\notag\\
                \label{alph>1-series-beta>0} &= e^{i\beta\theta}\left(1-e^{-2\pi i\beta}\right)\frac{1}{\beta}\left[1-\,_2F_1(1,\beta;1+\beta;\alpha^{-1}e^{i\theta})\right],
            \end{align}
            so long as $\beta\neq0$. This completes the proof of Theorem \ref{thm} in the case where $\beta \in \cc \setminus \zz_{\le 0}$. To handle the cases when $\beta\in\zz_{\leq0}$, note that 
            \begin{equation}
                \int_{\theta-2\pi}^{\theta}e^{i(\beta+k)t}\;dt =
                \begin{cases}
                    2\pi\qquad\beta+k=0,\\
                    0\qquad\;\;\beta+k\in\zz\setminus\{0\},
                \end{cases}
            \end{equation}
            since the bounds of integration align with the period of the exponential unless the exponent is 0. Thus when $\beta\in\zz_{\leq0}$,
            \begin{align}
                \int_{\del\dd} m_{\alpha,\beta,\theta} &= -i\sum_{k=1}^\infty \alpha^{-k}2\pi\delta_{\beta,-k}\label{kronecker-delta}
            \end{align}
            where $\delta_{\beta,-k}$ is the classical Kronecker delta function. This completes the proof of Theorem \ref{thm} in the case where $\beta \in \zz_{\le 0}$.  This completes the first part of the main result.\\
            
            Next consider $|\alpha|<1$.
            We proceed in a manner analogous to that of the $|\alpha|>1$ case, omitting details for the sake of brevity. It holds that
            \begin{equation}\label{mortini-rupp-series-id}
                \int_{\del\dd} m_{\alpha,\beta,\theta} =
                \begin{cases}
                    e^{i\beta\theta}\left(1-e^{-2\pi i\beta}\right)\sum_{k=0}^\infty \frac{\alpha^k}{\beta-k}e^{-ik\theta}\quad\qquad\beta\notin\nn,\\
                    2\pi i\alpha^\beta \qquad\qquad\qquad\qquad\qquad\qquad\qquad\;\;\,\beta\in\nn.
                \end{cases}
            \end{equation}
                
           An application of \eqref{hypgeo-def} shows that for $\beta\neq0$,
            \begin{align}
                \sum_{k=0}^\infty \frac{\alpha^k}{\beta-k}e^{-ik\theta} &= \frac{1}{\beta}\sum_{k=0}^\infty \frac{-\beta}{-\beta+k}\left(\alpha e^{-i\theta}\right)^k,\notag\\
                &= \frac{1}{\beta}\sum_{k=0}^\infty \frac{(1)_k(-\beta)_k}{(1-\beta)_k k!}\left(\alpha e^{-i\theta}\right)^k,\notag\\
                \label{alph<1-series-2F1} &= \frac{1}{\beta}\,_2F_1(1,-\beta;1-\beta;\alpha e^{-i\theta}).
            \end{align}
            Combining \eqref{mortini-rupp-series-id} and \eqref{alph<1-series-2F1} completes the proof of Theorem \ref{thm}.

        \subsection{Reconciling Methods}
            Note that the steps of the contour method described in \S\ref{contour-1} and the simplifications in \S\ref{core} still hold so long as $\Arg(\alpha)\neq\theta$, $|\alpha|\neq 0,1$ and $\Re(\beta)>0$. The nontrivial equations of the series method hold so long as $|\alpha|\neq1$ and $\beta\notin\zz_{\geq0}$. Thus under all these conditions one can write an identity for \eqref{core-integral} in the case where $0<|\alpha|<1$:
            \begin{align}
                e^{i\beta\theta}\left(1-e^{-2\pi i\beta}\right)\frac{1}{\beta}\,_2F_1(1,-\beta;1-\beta;\alpha e^{-i\theta}) &= 2\pi i \alpha^\beta + e^{i\beta\theta}\left(1-e^{-2\pi i\beta}\right)\int_0^1\frac{e^{\beta\ln{t}}}{t-\alpha e^{-i\theta}}\;dt,\label{contour-series-id}\\
                \frac{1}{\beta}\,_2F_1(1,-\beta;1-\beta;\alpha e^{-i\theta}) &= \frac{2\pi i \alpha^\beta}{e^{i\beta\theta}\left(1-e^{-2\pi i\beta}\right)} + \int_0^1\frac{e^{\beta\ln{t}}}{t-\alpha e^{-i\theta}}\;dt,\notag\\
                \frac{1}{\beta}\,_2F_1(1,-\beta;1-\beta;\alpha e^{-i\theta}) &= \frac{2\pi i \alpha^\beta}{e^{i\beta\theta}\left(1-e^{-2\pi i\beta}\right)} + \frac{1}{\beta} - \int_0^1 \frac{e^{(\beta-1)\ln{t}}}{1-\left(\frac{1}{\alpha}e^{i\theta}\right)t}\;dt,\label{contour-with-core-series-id}\\
                \int_0^1 \frac{e^{(\beta-1)\ln{t}}}{1-\left(\frac{1}{\alpha}e^{i\theta}\right)t}\;dt &= \frac{2\pi i \alpha^\beta}{e^{i\beta\theta}\left(1-e^{-2\pi i\beta}\right)} + \frac{1}{\beta}\left[1 - \,_2F_1(1,-\beta;1-\beta;\alpha e^{-i\theta})\right] \label{new-id-alph}
            \end{align}
            where the equality in \eqref{contour-series-id} follows from Lemma \ref{contour-lemma} and \eqref{alph<1-series-2F1}, and the equality in \eqref{contour-with-core-series-id} follows from \eqref{basically-2F1}.\\
                
            %It would seem possible to derive the identity
            %\begin{equation}
                %\int_0^1 %\frac{t^{\beta-1}}{1-zt}\;dt %= \frac{2\pi i %z^{-\beta}}{1-e^{2\pi i\beta}} + %\frac{1}{\beta}\left[1 - %\,_2F_1(1,-\beta;1-\beta;z^{-%1})\right]\notag,
            %\end{equation}
            %or some sufficient version of it, from the available hypergeometric identities; this type of result would have allowed us to use the contour method in the case $0<|\alpha|<1$. Exploring this would be an interesting direction for further study.

    \section{Computing the Example $\beta = \frac{m}{n}$.} \label{m/n-sec}
        Since the hypergeometric function gives the value of \eqref{main-integral} as an infinite series which is still difficult to explicitly evaluate, it is desirable to compute examples for which the hypergeometric function can be simplified more. We show this is the case when $\beta = \tfrac{m}{n}\in\qq\setminus\zz$, with $m\in\zz$, $n\in\nn$. We ignore the cases $\beta\in\zz$ since these are easily evaluated without need of the hypergeometric function.
        
        \begin{corollary}\label{b=rational}
            Let $\beta = \tfrac{m}{n}\in\qq\setminus\zz$. When $|\alpha|<1$, 
            \begin{equation}
                \int_{\del\dd} m_{\alpha,\beta,\theta} = \alpha^{-\frac{m}{n}}(1-e^{-2\pi i\frac{m}{n}}) \sum_{j=0}^{n-1}e^{\frac{2\pi ijm}{n}}\log\left(1-e^{\frac{2\pi ij}{n}}\sqrt[n]{\alpha e^{-i\theta}}\right)\notag
            \end{equation}
            and when $|\alpha|>1$
            \begin{equation}
                \int_{\del\dd} m_{\alpha,\beta,\theta} = (1-e^{-2\pi i\frac{m}{n}}) \left(\frac{n}{m}e^{i\frac{m}{n}\theta}+\alpha^{\frac{m}{n}}\sum_{j=0}^{n-1}e^{-\frac{2\pi ijm}{n}}\log\left(1-e^{\frac{2\pi ij}{n}}\sqrt[n]{\alpha^{-1}e^{i\theta}}\right)\right)\notag.
            \end{equation}
            \end{corollary}
    \begin{proof}
        From Theorem \ref{thm}, one has for $\beta\notin\zz$ that
        \begin{equation}\label{rational-case}
            \int_{\del\dd}m_{\alpha,\frac{m}{n},\theta} = 
            \begin{cases}
                e^{i\frac{m}{n}\theta}\left(1-e^{-2\pi i\frac{m}{n}}\right)\frac{n}{m}\,_2F_1(1,-\frac{m}{n};1-\frac{m}{n};\alpha e^{-i\theta}) \qquad\qquad\quad\; |\alpha|<1,\\
                e^{i\frac{m}{n}\theta}\left(1-e^{-2\pi i\frac{m}{n}}\right)\frac{n}{m}\left[1-\,_2F_1(1,+\frac{m}{n};1+\frac{m}{n};\alpha^{-1}e^{i\theta})\right] \qquad |\alpha|>1.
            \end{cases}
        \end{equation}
        Therefore the main difficulty in evaluating \eqref{rational-case} lies in computing 
        \begin{equation}\label{rational-case-params}
            _2F_1\left(1,\frac{m}{n};1+\frac{m}{n};z\right)
        \end{equation}
        for non-integral $\tfrac{m}{n}$ and for $0<|z|<1$.
            
        Since $\frac{m}{n}\notin\zz$, it is never the case that the parameter $c=1+\frac{m}{n}$ in \eqref{rational-case-params} is 0 or a negative integer. The hypergeometric series is therefore well defined, and using the definition \eqref{hypgeo-def} yields
        \begin{align}
            _2F_1\left(1,\frac{m}{n};1+\frac{m}{n};z\right):&= \sum_{k=0}^\infty \frac{(1)_k(\frac{m}{n})_k}{k!(1+\frac{m}{n})_k}z^k. \notag \\
            &= \notag \sum_{k=0}^\infty \frac{\frac{m}{n}}{k+\frac{m}{n}}z^k \\
            &= \notag m\sum_{k=0}^\infty \frac{1}{m+nk}z^k\\
            &= \frac{m}{z^{\frac{m}{n}}} \sum_{k=0}^\infty \frac{1}{m+nk}z^{\frac{m}{n}+k} =: G(z). \label{defining-G(z)}
        \end{align}
        Note also that since $0<|z|$, division by a fractional power of $z$ causes no issue. The particular choice of branch for defining the $n^\textnormal{th}$ root does not matter so long as the choice is consistent across the fractional powers (see remarks \ref{log-rmk} and \ref{nrt-branch-rmk}).
        
        Notice that the expression for $G(z)$ produces an even simpler expression for $G(z^n)$, given by 
        \begin{equation}\label{almostlog}
            G(z^n) = \frac{m}{z^m}\sum_{k=0}^\infty \frac{1}{m+nk}z^{m+nk} 
        \end{equation}
        On the other hand, for any branch of the logarithm with $\log(1)=0$ analytic in a ball of radius 1 at $z=1$, we have 
        \begin{equation}
             \log(1-z) = \sum_{k=1}^\infty -\frac{1}{k}z^k\notag
        \end{equation}
        whenever $|z| < 1$. The difference between the above expression and that in \eqref{almostlog} is that only terms of the form $z^{m+nk}$ for $k \in \nn \cup \{0\}$ appear in \eqref{almostlog}, whereas a $z^k$ term appears in the series for $\log(1-z)$ for every $k \in \nn$. To rectify this we express $G(z^n)$ as some series $\sum_{s=1}^\infty \frac{\delta_s}{s}z^s$, possibly with leading factors, where $\delta_s$ takes on the value $1$ whenever $n$ divides $s-m$ (so that $s$ is of the form $m + nk$ for some $k \in \nn$) and is $0$ otherwise.\\
            
        To find a suitable function $\delta_s$, recall that the sum of all of the $n^\textnormal{th}$ roots of unity is $0$ for $n > 1$. It is natural then that 
        \begin{equation}\label{kronecker}
            \delta_s = \frac{1}{n}\sum_{j=0}^{n-1}e^{\frac{2\pi i j(s-m)}{n}} =
            \begin{cases}
               1 &\quad \text{if} \hspace{0.15cm} n \mid (s-m)\\
               0 &\quad \text{otherwise} 
            \end{cases}
        \end{equation}
        is the desired function.  To see the validity of this claim, we first consider when $n\mid(s-m)$. In this case, we have $\frac{s-m}{n}\in\zz$, and so
        \begin{equation}\label{delta-s}
            \frac{1}{n}\sum_{j=0}^{n-1}e^{\frac{2\pi ij(s-m)}{n}} = \frac{1}{n}\sum_{j=0}^{n-1} 1 = 1
        \end{equation}
        since $\frac{j(s-m)}{n}\in\zz$. On the other hand, when $n \nmid (s-m)$, let $d:=\gcd(n,s-m)$ and define $\eta := \tfrac{n}{d}$. Note that $d<n$ else we have $n\mid(s-m)$, and thus $\eta>1$. Now
        \begin{equation}
            e^{\frac{2\pi i(s-m)}{n}} = \zeta_\eta^{\frac{s-m}{d}}
        \end{equation}
        where $\zeta_\eta$ is the first primitive ${\eta}^\textnormal{th}$ root of unity. Note also that because $d$ is the greatest common divisor of $s-m$ and $n$, then $\frac{s-m}{d}$ is coprime to $\eta$. This implies that $\zeta_\eta^{\frac{s-m}{d}}$ is another primitive ${\eta}^\textnormal{th}$ root of unity. Now  
        \begin{align}
            \delta_s &= \frac{1}{n}\sum_{j=0}^{n-1}e^{\frac{2\pi i j(s-m)}{n}},\notag\\
            &= \frac{1}{d\eta}\sum_{j=0}^{d\eta-1}\zeta_\eta^{\frac{s-m}{d}j},\notag\\
            \label{ds-d-rounds} &= \frac{1}{d\eta}\left(\sum_{j=0}^{\eta-1}\zeta_\eta^{\frac{s-m}{d}j} +\sum_{j=\eta}^{2\eta-1}\zeta_\eta^{\frac{s-m}{d}j} + \cdots + \sum_{j=(d-1)\eta}^{d\eta-1}\zeta_\eta^{\frac{s-m}{d}j}\right),\\
            \label{ds-d-times} &= \frac{1}{\eta}\sum_{j=0}^{\eta-1}\zeta_\eta^{\frac{s-m}{d}j},\\
           \label{ds-zero} &= 0.
        \end{align}
        The equality between \eqref{ds-d-rounds} and \eqref{ds-d-times} holds since $\zeta_\eta^{\frac{s-m}{d}j}=\zeta_\eta^{\frac{s-m}{d}j'}$ when $j\equiv j'\pmod{\eta}$.
        The final equality, \eqref{ds-zero}, holds since the sum over every $\eta^\textnormal{th}$ root of 1 is 0.
            
        One can therefore express
        \begin{align}
            G(z^n) = \frac{m}{z^m}\sum_{s=1}^\infty \frac{\delta_s}{s}z^{s}
        \end{align}
        since this series gives the same terms as the series in \eqref{almostlog}. To evaluate this series, note that for each $j$ in $\{0,\dots,n-1\}$, the series $\sum_{s=1}^{\infty}\left|\frac{e^{\frac{2\pi ij(s-m)}{n}}}{s}\right|z^s$ is a power series which converges absolutely for $|z| < 1$. Letting the value of this series be denoted $b_j$, we also note that $\sum_{j=0}^{n-1}b_j$ converges since it is a finite sum.
        We may therefore exchange the order of summation to get 
        \begin{equation}\label{betterformofG}
            \sum_{j=0}^{n-1}\sum_{s=1}^{\infty}\frac{e^{\frac{2\pi ij(s-m)}{n}}}{s}z^s = \sum_{s=1}^{\infty}\sum_{j=0}^{n-1}\frac{e^{\frac{2\pi ij(s-m)}{n}}}{s}z^s = G(z^n)
        \end{equation} 
        which allows evaluation of $G(z^n)$ by simplifying the left hand side of \eqref{betterformofG}:
        \begin{align}
            G(z^n) =  \frac{m}{z^m}\sum_{j=0}^{n-1}\sum_{s=1}^{\infty}\frac{e^{\frac{2\pi ij(s-m)}{n}}}{s}z^s &= \frac{m}{z^m}\sum_{j=0}^{n-1}e^{-\frac{2\pi ijm}{n}}\sum_{s=1}^{\infty}\frac{e^{\frac{2\pi ijs}{n}}}{s}z^s,\notag\\
            &= \frac{m}{z^m}\sum_{j=0}^{n-1}-e^{-\frac{2\pi ijm}{n}}\sum_{s=1}^{\infty}-\frac{1}{s}(e^{\frac{2\pi ij}{n}}z)^s,\notag\\
            &= \frac{m}{z^m}\sum_{j=0}^{n-1}-e^{-\frac{2\pi ijm}{n}}\log(1-e^{\frac{2\pi ij}{n}}z).
        \end{align}
        \begin{remark}\label{log-rmk}
           Notice that the only requirement of the branch of $\log$ we choose is that it is analytic in the ball of radius 1 at $z=1$, and that $\log(1)=0$.
        \end{remark}
        Finally, to come up with an expression for $G(z)$ as opposed to $G(z^n)$, simply substitute $z^{\frac{1}{n}}$ in the expression above, yielding
        \begin{equation}\label{rational-case-formula}
            G(z) = -\frac{m}{n}z^{-\frac{m}{n}}\sum_{j=0}^{n-1}e^{-\frac{2\pi ijm}{n}}\log(1-e^{\frac{2\pi ij}{n}}\sqrt[n]{z})
        \end{equation}
        whenever $|z| < 1$.\\
            
        \begin{remark}\label{nrt-branch-rmk}
            The choice of branch for $\sqrt[n]{\cdot}$ does not matter, so long as the choice is consistent across the expression for $G(z)$. To see this more clearly, rewrite
            \begin{equation}\label{symmetric-form}
                G(z) = -\frac{m}{n}\sum_{j=0}^{n-1}\left(e^{-\frac{2\pi ij}{n}}\frac{1}{\sqrt[n]{z}}\right)^m\log\left(1-e^{\frac{2\pi ij}{n}}\sqrt[n]{z}\right).
            \end{equation}
            This sum is symmetric over the $n^\textnormal{th}$ roots of $z$.
            Any branch of $\sqrt[n]{\cdot}$ must map an input $z$ to one of the $n$ possible roots $\omega$ of  $\omega^n=z$. The symmetry in \eqref{symmetric-form} implies that, no matter the branch chosen, this sum will always have the same terms.
        \end{remark}
        From \eqref{defining-G(z)} we know that $G(z) = \,_2F_1\left(1,\frac{m}{n};1+\frac{m}{n};z\right)$, and hence \eqref{rational-case-formula} allows us to conclude that for $\beta = \tfrac{m}{n}\in\qq\setminus\zz$ with $m\in\zz$, $n\in\nn$, 
            \begin{equation}\label{before-substituting-z's} _2F_1\left(1,\frac{m}{n};1+\frac{m}{n};z\right) = -\frac{m}{n}z^{-\frac{m}{n}}\sum_{j=0}^{n-1}e^{-\frac{2\pi ijm}{n}}\log(1-e^{\frac{2\pi ij}{n}}\sqrt[n]{z})
            \end{equation}
            for all $|z|<1$. Finally, when $|\alpha|<1$ we have $|\alpha e^{-i\theta}|<1$, so substituting $z=\alpha e^{-i\theta}$ into \eqref{before-substituting-z's} proves the first conclusion of Corollary \ref{b=rational}. Similarly, when $|\alpha|>1$, we have that $|\frac{1}{\alpha}e^{i\theta}|<1$, and hence substituting $z=\alpha^{-1}e^{i\theta}$ into \eqref{before-substituting-z's} proves the second conclusion.
            \end{proof} 
            
        \iffalse 
        \begin{corollary}\label{b=rational}
            For $\beta = \tfrac{m}{n}\in\qq\setminus\zz$, when $|\alpha|<1$
            \begin{equation}
                \int_{\del\dd} m_{\alpha,\beta,\theta} = \alpha^{-\frac{m}{n}}(e^{2\pi i \frac{m}{n}}-1) \sum_{j=0}^{n-1}e^{\frac{2\pi ijm}{n}}\log\left(1-e^{\frac{2\pi ij}{n}}\sqrt[n]{\alpha e^{-i\theta}}\right)\notag.
            \end{equation}
            and when $|\alpha|>1$
            \begin{equation}
                \int_{\del\dd} m_{\alpha,\beta,\theta} = (e^{2\pi i \frac{m}{n}}-1) \left(\frac{n}{m}e^{i\frac{m}{n}\theta}+\alpha^{\frac{m}{n}}\sum_{j=0}^{n-1}e^{-\frac{2\pi ijm}{n}}\log\left(1-e^{\frac{2\pi ij}{n}}\sqrt[n]{\alpha^{-1}e^{i\theta}}\right)\right)\notag.
            \end{equation}
            \end{corollary}
        \fi 
        
    \section{Differential Equation} \label{diffeq}
        A key feature of the hypergeometric equation 
        \begin{equation} \label{hypergeo-eq-og}
            z(1-z) \frac{d^2F}{dz^2} + (c-(a+b+1)z)\frac{dF}{dz} - ab F = 0
        \end{equation}
        is its regular singular points, and it is well-known that they are $0,1,\infty.$ Hence, one might wish to derive a second-order ordinary differential equation in the variable $\alpha$ for which $I(\alpha) = \int_{\partial \mathbb{D}} m_{\alpha,\beta,\theta}$ is a solution, and determine its regular singular points.

        \subsection{The case $|\alpha|>1$} \label{sec:diffeq-|a|>1}
            For $|\alpha|>1$, $\beta \notin \zz_{\leq0}$, the desired equation follows by relating \eqref{main-integral} to the hypergeometric function $F(z) = {_2F_1}(a,b,c;z)$, which is famously a solution of \eqref{hypergeo-eq-og}.
            \\
            Let $f(z) = {_2F_1}(1,\beta,1+\beta;z)$. Then $f$ solves the equation
            \begin{equation}\label{hypergeometric-old}
                z(z-1)\frac{d^2 f}{d z^2} + ((1+\beta)-(2+\beta)z) \frac{d f}{d z} - \beta f = 0.
            \end{equation}
            Consider the change in variables $\alpha = \frac{1}{z}e^{i \theta}$ (equiv. $z = \frac{1}{\alpha} e^{i \theta}$), and make the following necessary calculations:
            \begin{align}
                \frac{df}{dz} & = \frac{d\alpha}{dz} \frac{df}{d\alpha} = -\frac{1}{z^2} e^{i \theta} \frac{df}{d\alpha}  = -\alpha^2 e^{-i \theta} \frac{df}{d\alpha}, \notag \\
                \frac{d^2f}{dz^2} & = \frac{d\alpha}{dz} \cdot \frac{d}{d\alpha} \frac{df}{dz}\notag\\ 
                &= -\alpha^2 e^{-i\theta} \left(-\alpha^2 e^{-i \theta} \frac{d^2f}{d\alpha^2} -2\alpha e^{-i\theta} \frac{df}{d\alpha}\right)\notag\\
                &= \alpha^4 e^{-2i\theta} \frac{d^2f}{d\alpha^2} + 2\alpha^3 e^{-2i\theta} \frac{df}{d\alpha}.\notag 
            \end{align}
            By substituting into \eqref{hypergeometric-old}, notice that $f_*(\alpha) := f\left(\frac{1}{\alpha} e^{i \theta}\right)$ solves 
            \begin{equation}
                \alpha^{-1} e^{i \theta} \left(\alpha^{-1} e^{i \theta}-1\right) \left(\alpha^4 e^{-2i \theta} \frac{d^2f_*}{d\alpha^2} + 2\alpha^3 e^{-2i \theta} \frac{df_*}{d\alpha}\right) + \left((1+\beta)-(2+\beta)(\alpha^{-1} e^{i \theta})\right) \left(-\alpha^2 e^{-i \theta}\frac{df_*}{d\alpha}\right) - \beta f_* = 0,\notag
            \end{equation}
            which after some simplification becomes
            \begin{equation} \label{diffeq-f_*}
                p_2(\alpha) \frac{d^2f_*}{d\alpha^2}  + p_1(\alpha) \frac{df_*}{d\alpha} - \beta f_* = 0,
            \end{equation}
            where $p_2(\alpha) = \alpha^2 - \alpha^3 e^{-i \theta}, p_1(\alpha) = \alpha(\beta+4) - \alpha^2(\beta+3)e^{-i \theta}$.
            \\
            From Theorem \ref{thm}, $f_*(\alpha) = 1-kI(\alpha)$, with the abbreviation $k=\frac{\beta}{e^{i\beta\theta}(1-e^{-2\pi i\beta})}$, and we calculate the derivatives to be
            %TODO: justify chainrule & derivatives... dI/da probably fine since Morera's thm.
            \begin{align}
                \frac{df_*}{d\alpha} &= \frac{df_*}{dI} \frac{dI}{d\alpha} = -k \frac{dI}{d\alpha}, \notag \\
                \frac{d^2f_*}{d\alpha^2} &= -k \frac{d^2I}{d\alpha^2}. \notag
            \end{align}
            Substitution into $\eqref{diffeq-f_*}$ yields that $I(\alpha)$ solves the equation
            \begin{equation}
                 p_2(\alpha) \left( -k \frac{d^2I}{d\alpha^2} \right)  + p_1(\alpha) \left( -k \frac{dI}{d\alpha} \right) - \beta \left( 1 - k I \right) = 0,\notag
            \end{equation}
            or rather,
            \begin{equation}\label{hypergeo-I}
                p_2(\alpha) \frac{d^2I}{d\alpha^2} + p_1(\alpha) \frac{dI}{d\alpha} - \beta I = e^{i \beta \theta} (e^{-2\pi i\beta}-1).
            \end{equation}
            From \eqref{hypergeo-I} it is clear that the normalized coefficients $\frac{p_1(\alpha)}{p_2(\alpha)}\alpha$ and $\frac{-\beta}{p_2(\alpha)}\alpha^2$ are analytic in a neighborhood of $0$. Similarly, $\frac{p_1(\alpha)}{p_2(\alpha)}(\alpha-e^{i \theta})$ and $\frac{-\beta}{p_2(\alpha)}(\alpha-e^{i \theta})^2$ are analytic in a neighborhood of $e^{i \theta}$. These coefficients have poles at $0$ and $e^{i \theta}$, so these are regular singular points. To classify the point at infinity, let $x=1/\alpha$ and rewrite \eqref{hypergeo-I} in $x$. Akin to a previous change of variables, one has
            \begin{align}
                \frac{dI}{d\alpha} & = -x^2 \frac{dI}{dx}, \notag \\
                \frac{d^2I}{d\alpha^2} & = x^4 \frac{d^2I}{dx^2} + 2x^3 \frac{dI}{dx}, \notag
            \end{align}
            so that \eqref{hypergeo-I} becomes
            \begin{equation}
                p_2 \left( \frac{1}{x} \right) \left( x^4 \frac{d^2I}{dx^2} + 2x^3 \frac{dI}{dx}\right) + p_1\left(\frac{1}{x}\right) \left( -x^2 \frac{dI}{dx}\right) - \beta I = e^{i \beta \theta} (e^{-2\pi i\beta}-1),\notag
            \end{equation}
            or equivalently,
            \begin{equation}\label{hypergeo-inverted}
                q_2(x) \frac{d^2I}{dx^2} + q_1(x) \frac{dI}{dx} - \beta I = e^{i \beta \theta} (e^{-2\pi i\beta}-1),
            \end{equation}
            where $q_2(x) = x^2-xe^{-i \theta},$ $q_1(x) = -x(\beta+2) + (\beta+1)e^{-i \theta}.$ By a similar line of reasoning, the regular singular points of \eqref{hypergeo-inverted} are $x=0$ and $x=e^{-i \theta},$ so $\alpha=\infty$ and $\alpha=e^{i \theta}$ are both regular singular points of \eqref{hypergeo-I}.
            \\
            Thus equation \eqref{hypergeo-I}, for which \eqref{main-integral} is a solution, has precisely three regular singular points at $0, e^{i \theta}, \infty,$ reminiscent of \eqref{hypergeo-eq-og}. Any function satisfying a differential equation with three regular singular points may be expressed using the hypergeometric function, so this result supports the validity of the relationship derived.

        \subsection{The case $|\alpha|<1$} \label{sec:diffeq-0|a|<1}
            When $|\alpha|<1$, $\beta \notin \zz_{\geq0}$, one can proceed exactly as \S\ref{sec:diffeq-|a|>1} and make use of Theorem \ref{thm}. 
            \\
            Let $g(z)= {_2}F_1(1,-\beta,1-\beta,z).$ Then $g$ solves the equation
            \begin{equation}
                z(z-1) \frac{d^2g}{dz^2} + ((1-\beta) - (2-\beta) z) \frac{dg}{dz} + \beta g = 0.\notag
            \end{equation}
            The change of variables $\alpha=ze^{i \theta}$ gives that $g_*(\alpha) := g(\alpha e^{-i \theta})$ solves
            \begin{equation} \label{diffeq-g_*}
                r_2(\alpha) \frac{d^2 g_*}{d \alpha^2} + r_1(\alpha) \frac{dg_*}{d\alpha} + \beta g_* = 0,
            \end{equation}
            where $r_2(\alpha) = \alpha^2 - \alpha e^{i \theta},$ $r_1(\alpha) = (1-\beta) e^{i \theta} - (2-\beta) \alpha$. Theorem \ref{thm} gives $g_*(\alpha) = k I(\alpha)$, where again $k=\frac{\beta}{e^{i\beta\theta}(1-e^{-2\pi i\beta})}$. This scaling does not change the equation, so $I(\alpha)$ is also a solution of $\eqref{diffeq-g_*},$ with $g_*$ replaced with $I$. Just as before, one reasons that $\alpha=0,e^{i \theta}$ are regular singular points of this equation. In the variable $x=\frac{1}{\alpha},$ the equation \eqref{diffeq-g_*} written for $I$ is
            \begin{equation}
                s_2(x) \frac{d^2I}{dx^2} + s_1(x) \frac{dI}{dx} + \beta I = 0,\notag
            \end{equation}
            where $s_2(x)=x^2-x^3 e^{i \theta}, s_1(x) = \beta x - (1+\beta) e^{i \theta} x^2,$ of which the regular singular points are $x=0,e^{-i \theta}.$ Finally, one concludes that the regular singular points of the hypergeometric-like differential equation that $I(\alpha)$ solves are $0,e^{i \theta}, \infty$.
    
    \section*{Acknowledgements}
    The authors thank Mihai Putinar for helpful discussions and edits.
    
%\newpage
    \bibliographystyle{unsrt}
    \bibliography{main}

\end{document}